\documentclass[12pt]{amsart}
\usepackage{mathrsfs,eucal,amsmath,amssymb,latexsym,dsfont,stmaryrd}
\usepackage{enumerate}
\usepackage{tikz-cd}
\usepackage[all,2cell]{xy} \UseAllTwocells \SilentMatrices
\usepackage[T1]{fontenc}
\usepackage{xcolor}
\usepackage[pdftex]{hyperref}

\begin{document}
	
	\newcommand{\INVISIBLE}[1]{}

	\newtheorem{thm}{Theorem}[section]
	\newtheorem{lem}[thm]{Lemma}
	\newtheorem{cor}[thm]{Corollary}
	\newtheorem{prp}[thm]{Proposition}
	\newtheorem{conj}[thm]{Conjecture}
	
	\theoremstyle{definition}
	\newtheorem{dfn}[thm]{Definition}
	\newtheorem{question}[thm]{Question}
	\newtheorem{nota}[thm]{Notation}
	\newtheorem{conv}[thm]{Convenien}
	\newtheorem*{claim*}{Claim}
	\newtheorem{ex}[thm]{Example}
	\newtheorem{counterex}[thm]{Counter-example}
	\newtheorem{rmk}[thm]{Remark}
	\newtheorem{rmks}[thm]{Remarks}
	
	\def\labelenumi{(\arabic{enumi})}

	\newcommand{\aro}{\longrightarrow}
	\newcommand{\arou}[1]{\stackrel{#1}{\longrightarrow}}
	\newcommand{\RA}{\Longrightarrow}

	\newcommand{\mm}[1]{\mathrm{#1}}
	\newcommand{\bm}[1]{\boldsymbol{#1}}
	\newcommand{\bb}[1]{\mathbf{#1}}

	\newcommand{\bA}{\boldsymbol A}
	\newcommand{\bB}{\boldsymbol B}
	\newcommand{\bC}{\boldsymbol C}
	\newcommand{\bD}{\boldsymbol D}
	\newcommand{\bE}{\boldsymbol E}
	\newcommand{\bF}{\boldsymbol F}
	\newcommand{\bG}{\boldsymbol G}
	\newcommand{\bH}{\boldsymbol H}
	\newcommand{\bI}{\boldsymbol I}
	\newcommand{\bJ}{\boldsymbol J}
	\newcommand{\bK}{\boldsymbol K}
	\newcommand{\bL}{\boldsymbol L}
	\newcommand{\bM}{\boldsymbol M}
	\newcommand{\bN}{\boldsymbol N}
	\newcommand{\bO}{\boldsymbol O}
	\newcommand{\bP}{\boldsymbol P}
	\newcommand{\bY}{\boldsymbol Y}
	\newcommand{\bS}{\boldsymbol S}
	\newcommand{\bX}{\boldsymbol X}
	\newcommand{\bZ}{\boldsymbol Z}

	\newcommand{\cc}[1]{\mathcal{#1}}
	\newcommand{\ccc}[1]{\mathscr{#1}}

	\newcommand{\ca}{\cc{A}}
	
	\newcommand{\cb}{\cc{B}}
	
	\newcommand{\cC}{\cc{C}}
	
	\newcommand{\cd}{\cc{D}}
	
	\newcommand{\ce}{\cc{E}}
	
	\newcommand{\cf}{\cc{F}}
	
	\newcommand{\cg}{\cc{G}}
	
	\newcommand{\ch}{\cc{H}}
	
	\newcommand{\ci}{\cc{I}}
	
	\newcommand{\cj}{\cc{J}}
	
	\newcommand{\ck}{\cc{K}}
	
	\newcommand{\cl}{\cc{L}}
	
	\newcommand{\cm}{\cc{M}}
	
	\newcommand{\cn}{\cc{N}}
	
	\newcommand{\co}{\cc{O}}
	
	\newcommand{\cp}{\cc{P}}
	
	\newcommand{\cq}{\cc{Q}}
	
	\newcommand{\cR}{\cc{R}}
	
	\newcommand{\cs}{\cc{S}}
	
	\newcommand{\ct}{\cc{T}}
	
	\newcommand{\cu}{\cc{U}}
	
	\newcommand{\cv}{\cc{V}}
	
	\newcommand{\cy}{\cc{Y}}
	
	\newcommand{\cw}{\cc{W}}
	
	\newcommand{\cz}{\cc{Z}}
	
	\newcommand{\cx}{\cc{X}}

	\newcommand{\g}[1]{\mathfrak{#1}}
	
	\newcommand{\af}{\mathds{A}}
	\newcommand{\PP}{\mathds{P}}
	
	\newcommand{\GL}{\mathrm{GL}}
	\newcommand{\PGL}{\mathrm{PGL}}
	\newcommand{\SL}{\mathrm{SL}}
	\newcommand{\NN}{\mathds{N}}
	\newcommand{\ZZ}{\mathds{Z}}
	\newcommand{\CC}{\mathds{C}}
	\newcommand{\QQ}{\mathds{Q}}
	\newcommand{\RR}{\mathds{R}}
	\newcommand{\FF}{\mathds{F}}
	\newcommand{\DD}{\mathds{D}}
	\newcommand{\VV}{\mathds{V}}
	\newcommand{\HH}{\mathds{H}}
	\newcommand{\MM}{\mathds{M}}
	\newcommand{\OO}{\mathds{O}}
	\newcommand{\LL}{\mathds L}
	\newcommand{\BB}{\mathds B}
	\newcommand{\kk}{\mathds k}
	\newcommand{\bs}{\mathbf S}
	\newcommand{\GG}{\mathds G}
	\newcommand{\XX}{\mathds X}
	
	\newcommand{\WW}{\mathds W}
	\newcommand{\al}{\alpha}
	
	\newcommand{\be}{\beta}
	
	\newcommand{\ga}{\gamma}
	\newcommand{\Ga}{\Gamma}
	
	\newcommand{\om}{\omega}
	\newcommand{\Om}{\Omega}
	
	\newcommand{\vt}{\vartheta}
	\newcommand{\te}{\theta}
	\newcommand{\Te}{\Theta}

	\newcommand{\ph}{\varphi}
	\newcommand{\Ph}{\Phi}

	\newcommand{\ps}{\psi}
	\newcommand{\Ps}{\Psi}
	
	\newcommand{\ep}{\varepsilon}
	
	\newcommand{\vr}{\varrho}
	
	\newcommand{\de}{\delta}
	\newcommand{\De}{\Delta}

	\newcommand{\la}{\lambda}
	\newcommand{\La}{\Lambda}
	
	\newcommand{\ka}{\kappa}
	
	\newcommand{\si}{\sigma}
	\newcommand{\Si}{\Sigma}
	
	\newcommand{\ze}{\zeta}
	
	\newcommand{\fr}[2]{\frac{#1}{#2}}
	\newcommand{\vs}{\vspace{0.3cm}}
	\newcommand{\na}{\nabla}
	\newcommand{\pd}{\partial}
	\newcommand{\po}{\cdot}
	\newcommand{\met}[2]{\left\langle #1, #2 \right\rangle}
	\newcommand{\rep}[2]{\mathrm{Rep}_{#1}(#2)}
	\newcommand{\repo}[2]{\mathrm{Rep}^\circ_{#1}(#2)}
	\newcommand{\hh}[3]{\mathrm{Hom}_{#1}(#2,#3)}
	\newcommand{\modules}[1]{#1\text{-}\mathbf{mod}}
	\newcommand{\vect}[1]{#1\text{-}\mathbf{vect}}
	
	\newcommand{\Modules}[1]{#1\text{-}\mathbf{Mod}}
	\newcommand{\dmod}[1]{\mathcal{D}(#1)\text{-}{\bf mod}}
	\newcommand{\spc}{\mathrm{Spec}\,}
	\newcommand{\an}{\mathrm{an}}
	\newcommand{\NNo}{\NN\smallsetminus\{0\}}
	
	\newcommand{\pos}[2]{#1\llbracket#2\rrbracket}
	\newcommand{\poss}[2]{#1[\![#2_s]\!]}
	
	\newcommand{\lau}[2]{#1(\!(#2)\!)}
	\newcommand{\laus}[2]{#1(\!(#2_{s})\!)}

	\newcommand{\cpos}[2]{#1\langle#2\rangle}
	
	\newcommand{\id}{\mathrm{id}}
	
	\newcommand{\one}{\mathds 1}
	
	\newcommand{\ti}{\times}
	\newcommand{\tiu}[1]{\underset{#1}{\times}}
	
	\newcommand{\ot}{\otimes}
	\newcommand{\otu}[1]{\underset{#1}{\otimes}}
	
	\newcommand{\wh}{\widehat}
	\newcommand{\wt}{\widetilde}
	\newcommand{\ov}[1]{\overline{#1}}
	\newcommand{\un}[1]{\underline{#1}}
	
	\newcommand{\op}{\oplus}
	
	\newcommand{\lid}{\varinjlim}
	\newcommand{\lip}{\varprojlim}

	\newcommand{\mcrs}{\bb{MC}_{\rm rs}}
	\newcommand{\mclog}{\bb{MC}_{\rm log}}
	\newcommand{\mc}{\mathbf{MC}}
	\newcommand{\smc}{\mathbf{SMC}}
	\newcommand{\enu}{\bb{End}^{\rm u}}
	\newcommand{\en}{\mathbf{End}}
	\newcommand{\fmod}{\mathbf{Mod}_{\rm f}}

	\newcommand{\ega}[3]{[EGA $\mathrm{#1}_{\mathrm{#2}}$, #3]}

	\newcommand{\asts}{\begin{center}$***$\end{center}}
	\title[Connections on a relative punctured disk]{Logarithmic decomposition of connections on a relatively punctured disk}

	\author[P. T.  Tam]{Ph\d am Thanh T\^am} 	\address{Department of Mathematics, Hanoi Pedagogical University 2, Vinh Phuc, Vietnam; Institute of Mathematics, Vietnam Academy of Science and Technology, Hanoi, 		Vietnam}      \email{phamthanhtam@hpu2.edu.vn}

	
	\maketitle
	
	\begin{abstract} Let $R=\pos Ct$ be the ring of power series over an algebraically closed field $C$ of characteristic zero. We show that each connection on a finite flat $\lau Rx$-module is the sum of a regular singular connection and a diagonalizable $\lau Rx$--linear endomorphism when it admits a Turrittin-Levelt-Jordan form over $\lau Rx$. This decomposition is compatible with the limit of the logarithmic decompositions of the connections obtained by the reduction of $R$ modulo $t^{k}$ of a given connection.
	\end{abstract}

	\section{Introduction}
	A module with a connection over $\lau Rx$ is a finite $\lau Rx$--module equipped with a $R$--linear map obeying the Leibniz's rule  with respect to $\vartheta:=x\frac{d}{dx}$. We will call these objects by connections over $\lau Rx$. Denoted by $\mc(\lau Rx/R)$ the category of connections over $\lau Rx$ with horizontal morphisms.

	The structure of arbitrary connections over $\lau Cx$ is well understood in the work of Turrittin \cite{turrittin55}, Levelt \cite{Lev75} and Katz \cite{Kat87}. In the case of regular singular connections, the structure is described in detail by Manin \cite{manin65} and Deligne \cite{deligne87} which the canonical logarithmic lattices play an important tool in their works. In \cite[Theorem 4.5]{HdST22}, authors call these latices by Deligne-Manin lattices. In \cite{Lev75}, Levelt shows that each connection $(E,\na)$ over $\lau Cx$ decomposes into a sum of a semisimple connection and a nilpotent $\lau Cx$--linear  endomorphism. 
	We establish a decomposition $$\nabla=\nabla^{\rm rs}+P_\na$$ which is called the logarithmic decomposition for $\na$, where $\nabla^{\rm rs}$ is a regular singular connection on $E$, while $P_\na$ is a semisimple $\lau Cx$--linear endomorphism with eigenvalues in $x_s^{-1}C[x_s^{-1}]$ for some $s\in \NN$. The regular singular part allows us to define the Deligne-Manin lattice for $(E,\nabla^{\rm rs})$ which coincides with the \textit{canonical lattice} of $(E,\nabla)$ in the sense of \cite[Theorem 3.3.1]{malgrange96}.

	Recently, regular singular connections over $\lau Rx$ are studied by \cite{HdST22} and \cite{HdSTT22} for the case $R$ is a Henselian Noetherian local $C$-algebra. This work aims to establish the logarithmic decomposition for connections over $\lau Rx$  in the case $R$ is a complete discrete valuation domain with the maximal ideal $\g r$. It is the continuation works in \cite{HdST22}, we study the existence of the logarithmic decomposition of connections over $\lau Rx$. By reducing modulo $\g r^{\ell+1}$, each connection $(E,\nabla)$ in $\mc(\lau Rx/R)$ is associated to $(E,\nabla)|_\ell$ in 
	$\mc(\lau {R_\ell}x/R_\ell)$, where $R_\ell:=R/\g r^{\ell+1}$ is the reduction of $R$ modulo $\g r^{\ell+1}$. The connection  $(E,\nabla)|_\ell$ can be considered as an object in  $\mc(\lau Cx/C)$ with an action of $R_\ell$, and moreover, the logarithmic decomposition of $(E,\nabla)|_\ell$ is compatible with the action of $R_\ell$. Using the completeness of $R$, we then go back to $R$ by passing to the limit as $\ell\to \infty$. We establish the logarithmic decomposition for connections over $\lau Rx$ equipped with a Turrittin-Levelt-Jordan form which is compatible with taking limits along $\g r$--adic topology. Additionally, these connections have a  Deligne-Malgrange lattice which is compatible with the Turrittin-Levelt-Jordan decomposition. 
	
	The paper is organized as follows. In section \ref{section2--0}, we obtain the logarithmic decomposition  of connections over $\lau Cx$ in a canonical way by using Galois descent. In section \ref{section3--0}, we extend these results to the case of connections are equipped with an action of a local Artinian algebra. In Section \ref{section4--0}, Theorem \ref{prp-31.03--2} gives a condition for the existence of the logarithmic decomposition for connections over $\lau Rx$ which is compatible with taking limit along the $\g r$--adic topology (see Theorem \ref{thm.31.03--1}).


	\subsection*{Notation and conventions}
	\begin{enumerate}[(1)]\item In this text,  let us fix a subset  $\tau\subset C$ of representatives of the co-sets $C/\ZZ$ such that $\tau \cap \ZZ=\{0\}$.
		\item Let  $\lau Rx:=\pos Rx[x^{-1}]$ and the derivation $\vt\left(\sum\limits_{n\geq k} a_nx^n\right)=\sum\limits_{n\geq k} na_nx^n.$
		\item  For any $s\in\NN^*$, let $x_s$ be {\it a primitive $s^{th}$ root} of $x$. We extend $\vt$ to a derivation on $\lau R{x_s}$ and denote it by the same symbol $\vt.$
		%

		\item For each $f\in R$ and $r\in \NN^*$, let 
		$\mm J^\flat_{r}(f)=\begin{pmatrix}
			f&0&\cdots&0
			\\
			*&\ddots&\ddots&\vdots
			\\
			\vdots&\ddots&\ddots&0
			\\
			0&\cdots&*&f\end{pmatrix}$ with $*\in \{0,1\}.$ 
		\item For each $f=\sum\limits_{k< 0}f_kx^k+\sum\limits_{i\ge0}f_ix^i  \in \lau Rx$ with $f_i\in R$, let $\mm{p}(f):=\sum\limits_{k< 0}f_kx^k$.
	\end{enumerate}

	\section{Review of Levelt-Turrittin theory}\label{section2--0} 
	
	\subsection{Basic material} Let $(R,\g r)$ be a complete discrete valuation domain of $C$-algebras and $K$ be its the fraction field. Let us fix an isomorphism $R\simeq \pos Ct$.

	\begin{dfn} \label{section2--1}
		The {\it category of connections},  $\mc(\lau Rx/R)$, has for    
		\begin{enumerate}\item[{\it objects}]   are pairs $(E,\na)$ consisting of a finite $\lau Rx$-module and an $R$-linear endomorphism $\na:E\to E$ satisfying Leibniz's rule $\na(fe)=\vt(f)e+f\na(e)$ for all $f\in \lau Rx$ and $e\in E$, and  
			\item[{\it arrows}] from $(E,\na)$ to  $(E',\na')$ are  $\lau Cx$-linear maps $\ph:E\to E'$ such that $\na'\ph=\ph\na$.  
		\end{enumerate}
		The category of {\it logarithmic connections},    $\mclog(\pos Rx/R)$, has for     
		\begin{enumerate}\item[{\it objects}] are pairs $(\ce,\na)$ consisting of a finite   $\pos Rx$-modules and a $C$-linear endomorphism $\na:\ce\to\ce$ satisfying Leibniz's rule $\na(fe)=\vt(f)e+f\na(e)$ for all $f\in \ce$ and $e\in \ce$, and   
			\item[{\it arrows}] from $(\ce,\na)$ to $(\ce',\na')$ are $\pos Rx$-linear maps $\ph:\ce\to \ce'$ such that $\na'\ph=\ph\na$.  
		\end{enumerate}
	\end{dfn}

	We possess an evident $R$-linear functor 
	\begin{equation}\label{gama-func}
		\ga:\mclog (\pos Rx/R)\aro\mc (\lau Rx/R).
	\end{equation}

	\begin{dfn}\label{section2--2} An object $(E,\na)\in\mc(\lau Rx/R)$ is said to be {\it regular-singular} if it is isomorphic to a certain $\ga(\ce,\na)$ with $(\ce,\na)\in\mclog(\pos Rx/R)$.  Such object $(\ce,\na)$ is called a {\it logarithmic model of $(E,\na)$}. In case the model $\ce$ is, in addition, a {\it free} $\pos Rx$-module, we shall speak of a logarithmic {\it lattice}.
	\end{dfn}

	The full subcategory of $\mc(\lau Rx/R)$ whose objects are regular-singular shall be denoted by $\mcrs(\lau Rx/R)$.

	\begin{dfn}\label{dfn.20220716--01}
		Let $\bb{End}_R$ be the category whose   
		\begin{enumerate}\item[{\it objects}] are couples $(V,A)$ consisting of a finite $R$-module $V$ with $A\in\mm{End}_R(V)$,  and 
			\item[{\it arrows}] from $(V,A)$ to $(V',A')$ are $R$-linear morphisms $\ph:V\to V'$ with $A'\ph=\ph A$.  
		\end{enumerate}
	\end{dfn}

	\begin{ex}
		For each $(V,A)\in\bb{End}_R$, let $D_A:\pos R{x}\otu RV\to \pos R{x}\otu RV$ be the map defined by $D_A(f\ot v) = \vt f\ot v+f\ot Av.$  Then $\mm{eul}(V,A):=\big(\pos R{x}\otu RV,D_A\big)$ is an object in $\mclog(\pos Rx/R)$ which is called the Euler connection over $\pos Rx$ associated to $(V,A)$.
	\end{ex}

	\begin{dfn}\label{28.06.2021--6}
		Let $(\ce,\nabla)$ be an object in $\mclog(\pos Rx/R)$.
		\begin{enumerate}[(i)]
			\item 
			The natural $R$-linear map $\mm{res}_\na: \ce/(x)\aro\ce/(x)$ is called the {\it residue} of $\na$.  The map
			$\ov{\mm{res}}_\na : \ce/(\g r,x)\aro\ce/(\g r,x)$ is called the {\it reduced residue} of $\na$.  
			\item The  eigenvalues of  $\ov{\mm{res}}_\na$ are called the (reduced) {\it  exponents} of $\na$. The set of exponents will be denoted by $\mm{Exp}(\na)$.
		\end{enumerate}
	\end{dfn}
	
	\begin{dfn}\label{section2--3}
		A logarithmic model $\ce\in\mclog(\pos Rx/R)$ of $(E,\na)$ is call \textit{Deligne-Manin} with respect to $\tau$ if all the exponents of $\ce$ lie in $\tau$. 
	\end{dfn}

	Note that $\lau Rx$ is a principle ideal domain (see \cite[Lemma 2.12]{HdSTT22}). Then we obtain the following theorem.
	\begin{thm}\label{10.12.2021--2}Let $(E,\na)\in\mc(\lau Rx/R)$ be given. Then if  $E$  has no $R$-torsion then it is free over $\lau Rx$.
	\end{thm}
	\begin{proof}
		It is obvious to see that $E$ is $R$-flat, hence it is $\lau Rx$-flat since \cite[Theorem 8.18]{HdST22}. Because $\lau Rx$ is a principle ideal domain, so $E$ is $\lau Rx$-free.
	\end{proof}

	For a given finite $\lau Rx$-module $W$, the submodule 
	\[
	\begin{split}W_{\mathrm{tors}}&=\bigcup_{k=1}^\infty(0:t^k)_W
		=\{w\in W\,:\,\text{some power of $t$ annihilates $w$}\}
	\end{split}
	\] is stable under $\na$. Noetherianity assures that $W_{\mm{tors}}=(0:t^\ell)_W$ for some $\ell\in \NN$. Let 
	\[
	\begin{split}
		r_W&=\min\{k\in\NN\,:\,t^kW_{\mm{tors}}=0\}=\min\{k\in\NN\,:\,W_{\mm{tors}}=(0:t^k)_W\}.
	\end{split}
	\]

	Denote by $\mc^{\mm{o}}\big(\lau Rx/R\big)$ the full subcategory of $\mc(\lau Rx/R)$ whose objects are connections over $\lau Rx$ having no $R$--torsion. By Theorem \ref{10.12.2021--2}, a connection $(N,\na)$ over $\lau Rx$ belongs to $\mc^\circ(\lau Rx/R)$ iff $(0:t)_N=0$. In particular, if $\mm r_N=0$ then $N\in \mc^\circ(\lau Rx/R)$. 
	
	The following is an analog of \cite[Proposition 2.13]{HdSTT22} for the connections over $\lau Rx$. 
	\begin{prp}\label{prp.202303.01}
		Each object in $\mc(\lau Rx/R)$ is a quotient of a certain object in the category $\mc^{\mm{o}}\big(\lau Rx/R\big)$.   
	\end{prp}
	\begin{proof}

		Let $(M,\na)\in\mc(\lau Rx/R)$. Consider the case $r_M=1$. We have $(\g rM)_{\mm{tors}}=0$ because if $tm\in (\g rM)_{\mm{tors}}$ then $m\in M_{\mm{tors}}$. Let $q:M\twoheadrightarrow Q$ be the quotient by $\g rM$. Then $Q$ is an object in $\mc(\lau Cx/C)$, hence it is $\lau Cx$--free. It is obvious that $\wt Q:=(\lau Rx\otu{\lau Cx}Q, \na)$ is an object in $\mc^\circ(\lau Rx/R)$. Let $\wt Q\twoheadrightarrow Q$ be the arrow in $\mc(\lau Rx/R)$ defined by taking modulo $\g r$. According to \cite[Lemma 1.2.5]{andre09}, the following diagram commutes with exact rows: 
		\[
		\xymatrix{0\ar[r]& \g rM\ar[r]&M\ar[r]^q&Q\ar[r]&0\\ 0\ar[r]& \g rM\ar[u]^\sim\ar[r]&\wt M\ar@{}[ru]|\square\ar[u] \ar[r]&\ar@{->>}[u] \wt Q\ar[r]&0.}
		\]
		The right square is Cartesian, so $\wt M\to M$ is surjective. Note that $\wt M_{\mm{tors}}=0$ because $(\g rM)_{\mm{tors}}=\wt Q_{\mm{tors}}=0$. Hence, $\wt M\in\mc^\circ(\lau Rx/R)$. 
		
		Consider $r_M>1$ and let $N=(0:t)_M$. Let $q:M\twoheadrightarrow Q$ be the quotient by $N$. By definition, $t^{r_M-1}Q_{\mm{tors}}=0$, therefore $r_Q\le r_M-1$. By induction, there exists $\wt Q\in\mc^\circ(\lau Rx/R)$ and a surjection $\wt Q\twoheadrightarrow Q$. According to \cite[Lemma 1.2.5]{andre09} again, we obtain the commutative following diagram with exact rows
		\[
		\xymatrix{0\ar[r]& N\ar[r]&M\ar[r]^q&Q\ar[r]&0\\ 0\ar[r]& N\ar[u]^\sim\ar[r]&\wt M\ar@{}[ru]|\square\ar[u] \ar[r]&\ar@{->>}[u]\wt Q\ar[r]&0.}
		\]
		Similarly, $\wt M\to M$ is surjective.  Because $r_N=1$ and $\wt M_\mm{tors}= N$, hence $r_{\wt M}=1$. By using previous step, there exists a surjection $\wt M_1\twoheadrightarrow\wt M$ with $\wt M_1\in\mc^\circ(\lau Rx/R)$. By taking the composition, we obtain that $\wt M_1\twoheadrightarrow M$ is surjective.
	\end{proof}

	\subsection{Jordan-Levelt decompositions}\label{section2--01} Let ${\rm dLog}:= \big\{u^{-1}\vt u \,:\,u\in\lau Cx\setminus\{0\}\big\}$. Then
	\[\begin{split}
		{\rm dLog}&=\left\{\sum_{n\ge0}a_nx^n\in\pos Cx\,:\,a_0\in\ZZ\right\}.
	\end{split}\]
	
	For each $f\in\lau Cx$, let us define $\g L_f=(\lau Cx \bb f,\mm d_f)$ 	by $\mm d_f(\bb f)=f\bb f$.
	Note that $\g L_f\simeq \g L_g$ if and only if $f-g\in{\rm dLog}$. Let $\XX$ stand for the set of isomorphism classes of rank one connections and endow it with the group structure obtained by the tensor product, it follows that 
	\[
	\fr{\lau Cx}{{\rm dLog}}\aro \XX,\quad f+{\rm dLog}\longmapsto \g L_f
	\] 
	is  a group isomorphism. Furthermore, $C[x^{-1}]/\ZZ\arou\sim \XX$ is isomorphic which is defined by $f+\ZZ\longmapsto \g L_f$. 	For each   $f\in C[x^{-1}]$, denote by $[f]$ for its class in  $C[x^{-1}]/\ZZ$. Then, there exists $c_f\in\tau$ such that 	$c_f+\mm p(f)$  is a representative in $C[x^{-1}]$ of $[f]$.  We say that $\g L_f$ is \textit{a connection of type $[f]$}.

	For each $s\in\NN^*$, consider a new category $\mc_s(\lau  C{x_s}/C)$ whose objects are couples $(E,\na)$ consisting of a finite dimensional $\lau C{x_s}$--vector space $E$ equipped with a connection $\na$ over $\laus Cx$ with respect to $\vt$. 
	Arrows are defined in the evident way.  There is a natural functor 
	\[(-)_{(s)}:
	\mc(\lau  C{x}/C)\aro\mc_s(\lau  C{x_s}/C),
	\]
	which, on the level of objects, is $(E,\na)_{(s)}=(E_{(s)},\na)$ 	with $E_{(s)}=\lau C{x_s}\otu{\lau Cx} E$ and the connection obeys the rule $	\na(g\ot m)=\vt(g)\ot m+g\ot\na(m)$, where the derivation $\vt$ on $\laus Cx$ is defined by $\vt= s^{-1}x_s\frac{d}{dx_s}$

	In parallel to Definitions \ref{section2--1} and \ref{section2--2}, we have:
	\begin{dfn}\label{16.12.2021--1}
		We define the category $\mc_{\log,s}(\pos  C{x_s}/C)$ having as objects couples $(\cm,\na)$ where $\cm$ is a finite $\pos C{x_s}$--module and $\na$ is a logarithmic connection over $\pos Rx$ with respect to $\vt$ and arrows between objects are defined in the evident way. 
	\end{dfn}
	
	Consider the restriction functor 
	\begin{equation}\label{gama-func2}
		\ga_s:\mc_{\log,s}(\pos C{x_s}/C)\to\mc_{s}(\lau  C{x_s}/C).
	\end{equation}

	The set
	\[\mathds X_s:=\left\{\begin{array}{c}\text{ group of isomorphism classes of objects}\\ \text{of rank one in $\mc_s(\lau C{x_s}/C)$}\end{array}\right\}\]
	is isomorphic to $C[x_s^{-1}]/s^{-1}\ZZ \arou\sim \lau C{x_s} / {\rm dLog}_s \arou\sim
	\XX_s$, where 
	\[\begin{split}
		{\rm dLog}_s&= \left\{\sum_{n\ge0}a_nx_s^n\in\pos C{x_s}\,:\,a_0\in s^{-1}\ZZ\right\}.\end{split} \]

	\begin{dfn}
		A connection $(E,\na)\in\mc(\lau Cx/C)$ is called {\it diagonalizable}   over $\lau C{x}$ if it is a direct sum of objects in $\XX$. The connection $(E,\na)$ is called \textit{semi-simple} if there exists $s\in\NN$ such that $(E_{(s)},\na)$ is diagonalizable. 
	\end{dfn}

	The following theorem is reminiscent of the Jordan-Chevalley decomposition for a linear operator. The detail proof see {\cite[Section 1,  Theorem I]{Lev75}}.
	\begin{thm}[{Jordan-Levelt} decomposition] \label{section2--4} 
		Let $(E,\na)\in\mc\big(\lau Cx/C\big)$. Then, there exists a unique decomposition $\na=S_\na+N_\na$ 
		satisfying the following conditions:
		\begin{itemize}
			\item[(i)] $(E,S_\na)$ is a semi-simple connection in $\mc\big(\lau Cx/C\big)$. 
			\item[(ii)] $N_\na:E\aro E$ is a $\lau  Cx$--linear nilpotent endomorphism.
			\item[(iii)] $S_\na$ and $N_\na$ commute. \qed
		\end{itemize}
	\end{thm}

	\begin{dfn}\label{section2--5} The {\it Turrittin index} of $(E,\na)$ is the smallest number $s\in \NN$ such that the connection $(E_{(s)},S_\na)$ is diagonalizable. We say that $(E,\na)$ is \textit{unramified} if its Turrittin index equals $1$.
	\end{dfn}

	\subsection{Turrittin-Levelt-Jordan decomposition}
	\textit{From on now}, denote by $s\in\NN$ for the Turrittin index of $(E,\na)$. Because $S_\na$ is diagonalizable,  so $(E_{(s)},S_\na)$ is a direct sum of connections over $\laus Cx$ of rank one. Hence, there exists a  finite set $\Ph\subset C[x_s^{-1}]/s^{-1}\ZZ$ and a decomposition 
	$(E_{(s)} ,S_{\na})=\bigoplus\limits_{\varphi\in\Phi} (E_{\ph},S_\ph)$
	in $\mc_s(\laus Cx/C)$, where $(E_{\ph},S_\ph)$ is the direct sum of the subconnections of type $\ph$ of $(E_{(s)} ,S_{\na})$. Notice that $S_\nabla$ and $N_\nabla$ commute, so $E_{\ph}$ is closed under the action of $S_\nabla, N_\nabla$ and $\na$. 
	Hence, we obtain the following decomposition
	\begin{equation}\label{TLJdecom--1}
		(E_{(s)} ,{\na})=\bigoplus\limits_{\varphi\in\Phi} (E_{\ph},\na_\ph)
	\end{equation}	
	in $\mc_s(\laus Cx/C)$ which is called the \textit{the Turrittin-Levelt-Jordan decomposition} (TLJ-decomposition) of $(E_{(s)},\na)$, see \cite[Theorem 16.1.2]{andre-baldassarri-cailotto20}. In addition, $(E_{\ph},\na_\ph)$ is called  the  \textit{isotypical component of type $\ph$} of $(E_{(s)},\na)$. For each $\ph\in\Ph$, denote by $\mm p(\ph)=\mm p(f_\ph)$, where $f_\ph$ is a representative in $\laus Cx$ of $\ph$.

	\begin{rmk}\label{21.03.2022--1} The set $\mm{Irr}(E,\na)=\big\{\mm p(\ph)\in \lau C{x_s} / {\rm dLog}_s:\ph \in \Ph\big\}$ is defined \textit{uniquely} and depends only on $(E,\na)$, see \cite[Th\'eor\`eme 2.3.1]{andre07}. We say that $\mm{Irr}(E,\na)$ is \textit{the set of irregular values of $\na$}.
	\end{rmk}
	

	\begin{lem}[{\cite[Section 4, Lemma]{Lev75}}]\label{section2--7} Let $\ph\in\Ph$ and $f_\ph$ be its representative in $\lau  C{x_s}$. There exists $q\in\NN$ such that $\ker_{\laus Cx}\big(\na -f_\ph \big)^{q}$ spans the $\lau  C{x_s}$--space $E_{\ph}$. In addition, $E_{\ph}$ is stable under $\na$, $S_\na$ and $N_\na$. \qed
	\end{lem}

	\begin{ex}\label{ex.202208-13--01}
		Let $(E,\na)\in \mcrs(\lau Cx/C)$. According to \cite[Corollary 4.3]{HdST22}, $(E,\na)\simeq \ga eul(A,V)$ for some $(A,V)\in\bb{End}_C$ with $\spc(A)\subset \tau$. Each $\vr\in \spc(A)$, let ${A_\vr}=A|_{\bb G(A,\vr)}$. The TLJ-decomposition  of $(E,\na)$ is $$(E,\na)=\bigoplus\limits_{\vr\in \spc(A)}\ga eul\big(\bb G(A,\vr),{A_{\vr}}\big).$$
	\end{ex}

	\begin{lem}\label{lem.principal-part} Let $(E,\na)$ be a diagonalizable connection over $\lau Cx$ and have only a isotypical component. Assume that $\mm{Mat}(\na, \bb v)=\mm{diag}(g_{1},\cdots,g_{\mm r}),$
		where $\bb v$ is a $\lau Cx$-basis of $E$ and $g_{1},\cdots,g_{\mm r}\in \lau Cx$. Then $\mm p(g_{1}),\cdots,\mm p(g_\mm r)$ are invariant under base change.
	\end{lem}
	\begin{proof} According to \cite[Section 6a]{Lev75}, there is a $\lau Cx$-basis  $\bb e$ of $E$ such that $\mm{Mat}(\na, \bb e)=\mm{diag}(f,\cdots,f)$. Let $Q=\big(q_{ij}\big)_{\mm r}$ be the matrix of the base change from $\bb e$ to $\bb v$ and $\la_{i}=g_{i}-f$ for each $i$. Since $\mm{Mat}(\na, \bb v)=f\mm I_{\mm r}+Q^{-1}\vt(Q)$, we have $$Q^{-1}\vt(Q)=\mm{diag}(\la_{1},\cdots,\la_{\mm r}).$$ 	This shows that $\vt(Q)=Q\mm{diag}(\la_1,\cdots,\la_{\mm r})$. Hence, $\vt(q_{ij})=\la_{j} q_{ij}$ for all $1\le i,j\le\mm r$. 		For each $1\le i\le\mm r$, because $Q\in\mm{GL}_{\mm r}(\laus Cx)$ so there exists $1\le k\le\mm r$ such that $q_{ik}\ne 0$. Therefore, $\la_{i}\in\pos Cx$, and hence, $\mm p(g_{i})=\mm p(f)$ for all $1\le i\le\mm r$.
	\end{proof}
	
	\begin{rmk}\label{TLJ-form-2402} Assume that $\Ph=\{\ph_1,...,\ph_\ell\}$. There is a $\lau C{x_s}$--basis $(\bb e)$ of $E_{(s)}$ such that $$\mm{Mat}(\na,\bb e)=\mm{diag}\big(\mm J^\flat_{r_{\ph_1}}(f_{\ph_1}),...,\mm J^\flat_{r_{\ph_\ell}}(f_{\ph_\ell})\big).$$ 
	\end{rmk}

	\subsection{Logarithmic decomposition}
	Let us now go further in Levelt's theory which allows us to establish the following decomposition.

	\begin{thm}[Logarithmic decomposition]\label{section2--12}
		Let $(E,\na)\in \mc(\lau Cx/C)$ and $s\in\NN$ be its Turrittin index. Then there exists a unique decomposition
		$\na=\na^{\rm rs} +P_\na$ 
		having the following properties:
		\begin{itemize}
			\item[(i)] $(E,\na^{\rm rs})$ is an object in $\mcrs\big(\lau Cx/C\big)$.
			\item[(ii)] $P_\na$ is a semisimple endomorphism of $E$ with all eigenvalues in $\laus Cx$.
			\item[(iii)] $(E,\na^{\rm rs})$ admits a TLJ-decomposition $(E,\na^{\rm rs}) = \bigoplus\limits_{\vr\in \Sigma}(E_\vr,\na_\vr)$ over $\lau Cx$, where $\Sigma\subset C/\ZZ$. Moreover, for each ${\vr\in \Sigma}$, we have
			\begin{itemize}
				\item[(a)] $P_\na(E_\vr)\subset E_\vr$. 
				\item[(b)] There exists an $\laus Cx$--basis $(\bb e_\vr)$ of $\laus Cx\otu{\lau Cx}E_{\vr}$, $c_\vr\in C$ and $f_{\mm \vr 1},...,f_{\mm \vr \ell}\in x_s^{-1}C[x_s^{-1}]$ such that
				\[
				\begin{cases}
					\mm{Mat}(\na^{\rm rs}|_{E_{\vr}},\bb e_\vr)=\mm J^\flat_{\mm r_\vr}(c_\vr)\\
					\mm{Mat}({\na}|_{E_{\vr}},\bb e_\vr)=\mm{diag}(\mm J^\flat_{\mm r_{\mm \vr1}}(c_\vr+f_{\mm \vr 1}),...,\mm J^\flat_{\mm r_{\mm \vr \ell}}(c_\vr+f_{\mm \vr \ell})).
				\end{cases}
				\]
			\end{itemize}  
		\end{itemize}
	\end{thm}
	\begin{proof} \textit{Existence.} According to \eqref{TLJdecom--1}, the connection $(E,\na)$ admits a TLJ-decomposition  $(E_{(s)},\na)=\bigoplus\limits_{\ph\in\Ph} 
		(E_{\ph},\na_{\ph})$
	over $\laus Cx$. For each $\ph\in\Ph$, let $c_\ph\in C$  such that  $c_\ph+\mm p(\ph)$ is a representative of $\varphi$. Let us {\it define}  on the $\lau C {x_s}$--space  $E_{\ph}$ a connection by ${\na}_{\ph}^{\mm{rs}} := \na_{\ph} - \mm p(\ph)\id_{E_{\ph}}$. Consider the Levelt-Jordan decomposition $\na=S_\na+N_\na$ of $\na$ and denote by $S_\ph=S_\na|_{E_\ph}$. According to \cite[Section 6a]{Lev75}, there exists a $\laus Cx$--basis $(\bb e_\ph)$ of $E_{\ph}$ such that $\mm{Mat}(S_\ph,\bb e_\ph) =(c_\ph+\mm p(\ph))\mm I_{r_\ph}$ and $\mm{Mat}(N_\na, \bb e_\ph)=\mm J^\flat_{\mm r_\ph}(0)$.
	It is clearly that the $C$-subspace $C \bb e_\ph$ of $E_{\ph}$ spanned by $(\bb e_\ph)$ is invariant under $N_\na$. Hence, $\mm{Mat}({\na}_{\ph}^{\mm{rs}},\bb e_\ph) = c_\ph\mm I_{r_\ph}+N_{\ph}$ which is a matrix in $\mm{M}_{r_\ph}(C)$.

	Let $\ze$ be a primitive $s^\mm{th}$ root of unity and   $\si:\lau C{x_s}\to\lau C{x_s}$ be the automorphism determined by    $x_s\mapsto\ze x_s$. Note that the Galois group $\mm G$ of the extension $\lau C{x_s}/\lau Cx$ is generated by $\si$. 
		Let us also abuse notation and write $\si$ for the map $\si\ot \id$ from $E_{(s)}$ to itself. 
		Clearly, $\na\circ\si=\si\circ\na, S_\na\circ\si=\si\circ S_\na$, and $	\si E_{\ph} =   E_{\si(\ph)}$. Now, $\mm p(\si\ph)=\si\mm p(\ph)$, and  hence $
		\si\circ\bigoplus\limits_{\ph\in\Ph}{\na}_{\ph}^{\rm rs}  =\bigoplus\limits_{\ph\in\Ph}{\na}_{\ph}^{\rm rs} \circ \si$.
	Note that $E=\big\{e\in E_{(s)}:\si e=e \text{ for all } \si\in \mm G\big\}$. Let $\na^{\rm rs}$ is the restriction of $\bigoplus\limits_{\ph\in\Ph}{\na}_{\ph}^{\rm rs}$ to $E$ and $P_\na:=\na-\na^{\rm rs}$. We see that $(E,\na^{\rm rs})$ is an object in $\mcrs\big(\lau Cx/C\big)$. Indeed, let $\ce_s$ is a Deligne-Manin lattice of $E_{(s)}$. According to \cite[Theorem 7.8]{HdST22}, $\si$ is extended to an action on $\ce_{s}$ by the functor \[\ga_s:\mc_{\log,s}(\pos C{x_s}/C)\to\mc_{s}(\lau  C{x_s}/C)\]
	which is compatible to the action on $E_{(s)}$. Let $\ce=\big\{e\in \ce_{s}:\si e=e \text{ for all } \si\in \mm G\big\}$. We obtain that $(\ce,\na^{\rm rs})$ is a Deligne-Manin lattice of $(E,\na^{\rm rs})$. By construction, it is obvious to see that the  $\lau Cx$--linear endomorphism $P_\na=\na-{\na}^{\rm rs}$ is semisimple with all eigenvalues in  $x_s^{-1}C[x_s^{-1}]$.

	Example \ref{ex.202208-13--01} shows that there exists $\Sigma\subset C/\ZZ$ and a TLJ-decomposition $$(E,\na^{\rm rs}) = \bigoplus\limits_{\vr\in \Sigma}(E_\vr,\na_\vr)$$ of $(E,\na^{\rm rs})$ over $\lau Cx$. Note that $E_\vr=\{e\in \bigoplus\limits_{\ph\in\Ph:[c_\ph]=\vr} E_{\ph}\,:\,\si e=e\}.$  
	Because $\mm p(\si\ph)=\si\mm p(\ph)$, we obtain that $E_\vr$ is invariant under $\mm G$		for each $\vr\in \vr$. Clearly, each $E_\vr$ is invariant under $P_\na$. In addition, the last assertion is clearly.

		\textit{Uniqueness.} Assume that $\na=\dot\na +\dot P$ satisfies the properties (i)--(iii).	According to Example \ref{ex.202208-13--01}, there exists $\Ga\subset C/\ZZ$ and TLJ-decomposition $(E,\dot \na) =  \bigoplus\limits_{\ga\in \Ga }E_{\ga}$ of $(E,\dot \na)$ over $\lau Cx$. For each $\ga\in \Ga$, there exist a $\laus Cx$--basis  $(\bb e_{\ga})=\{e_{\ga}^1,\cdots,e_{\ga}^{\mm r_{\ga}}\}$ of $\laus Cx\ot E_{\ga}$ such that 
		$\mm{Mat}(\dot\na,\bb e_{\ga})=\mm J^\flat_{\mm r_{\ga}}(c_{\ga})$ and $\mm{Mat}(P_{\na}|_{E_{\ga}},\bb e_\ga)=\mm{diag}(f_\ga^1,...,f_\ga^{\mm r_\ga})$, 		where $c_\ga\in C$, ${\mm r_\ga}=\dim_{\lau Cx}{E_{\te}}$ and $f_\ga^1,...,f_\ga^{\mm r_\ga}\in x_s^{-1}C[x_s^{-1}]^{\mm r_\ga}$. 	Let us define the connection $\dot S_{\ga}:E_{\ga}\to E_{\ga}$ by declaring 
		\begin{equation}\label{14.03.2022-01}
			\mm{Mat}(\dot S_{\ga},\bb e_{\ga})=c_{\ga}\mm I_{\mm r_{\ga}}+\mm{Mat}(\dot P,\bb e_{\ga})=\mm{diag}(c_{\ga}+f_{\ga}^1,...,c_\ga+f_\ga^{\mm r_\ga})
		\end{equation}
		which is a diagonalizable connection over $\lau Cx$. Note that  $\dot N_{\ga}=\na|_{E_{\ga}}-\dot S_{\ga}$ is a $\laus Cx$--linear endomorphism of $E_{\ga}$ satisfying $$\mm{Mat}(\dot N_{\ga},\bb e_{\ga})=\mm{Mat}(\dot\na,\bb e_{\ga})-c_{\ga}\mm I_{\mm r_{\ga}}=\mm J^\flat_{\mm r_{\ga}}(0)$$ which is a matrix in $\mm{Mat}_{\mm r_{\ga}}(C),$ hence $\dot N_{\ga}$ is nilpotent. 
		For each $e\in E_{\ga}$, there exist $a_1,\cdots,a_{\mm r_{\ga}}\in \laus Cx$ such that $e=a_1e_{\ga}^1,\cdots+a_{\mm r_{\ga}}e_{\ga}^{\mm r_{\ga}}$. By computing, we get 
		$$\dot S_{\ga}\dot N_{\ga}(e)=\sum_{i=1}^{{\mm r_{\ga}}-1}\big(\vt(a_i)+a_if_{\ga}\big)e_{\ga}^{i+1}=\dot N_{\ga}\dot S_{\ga}(e),$$
		hence $\dot S_{\ga}\dot N_{\ga}=\dot N_{\ga}\dot S_{\ga}$.
		
		Hence, we obtain that $\na|_{E_{\ga}}=\dot S_{\ga}+\dot N_{\ga}$ is the Jordan-Levelt decomposition of $\na|_{E_{\ga}}$. Let $\dot S$ be the connection on $E$ such that $\dot S|_{E_\ph}=\dot S_\ph$ for each $\ph\in \ga$ and $\dot N=\na-\dot S$. Then 
		$\na=\dot S+\dot N$ 		is the Jordan-Levelt decomposition of $\na$. According to Theorem \ref{section2--4}, $S_\na=\dot S$ and $N_\na=\dot N$. Therefore, $S_\na$ preserves each $E_{\ga}$ and so does $N_\na$. As a conclusion, $
		S_\na|_{E_{\ga}}= S_{\ga}$
		for all ${\ga}\in \Ga$. By combining with  \eqref{14.03.2022-01}, we obtain that $P_{\na}|_{E_{{\ga},(s)}}= \dot P|_{E_{{\ga},(s)}}$. Hence, $\dot P =P_\na$, and $\dot\na=\na^{\rm rs}$.
	\end{proof}

	The condition Theorem \ref{section2--12}.(iii) is necessary. The following is an illustration.
	\begin{ex}\label{ex.20220812--01}
		Consider $E=\lau Cx\mm e_1\op \lau Cx\mm e_2$ and $\na: E\aro E$ given by 
		\[\na(\mm e_1)=\dfrac{1}{x}\mm e_1+\mm e_2; \quad \na(\mm e_2)=-\dfrac{1}{x}\mm e_2.\]		
		
		By \cite[Section 2, Lemma]{Lev75}, there exists $G=\begin{pmatrix}
			1&g_{1}\\
			g_{2}&1
		\end{pmatrix}\in \mm{GL}_2(\lau Cx)$ and $a_{1},a_{2}\in x\pos Cx$ such that		
		\[\mm{diag}(\dfrac{1}{x}+a_1, -\dfrac{1}{x}+a_2)
		=
		G^{-1}\begin{pmatrix}
			\dfrac{1}{x}&0\\
			1&-\dfrac{1}{x}
		\end{pmatrix}G + G^{-1}\vt(G).\]

		Let $H=\mm{diag}(\mm{e}^{-\int x^{-1}a_1dx}, \mm{e}^{-\int x^{-1}a_2dx})$, we have $H^{-1}=\mm{diag}(\mm{e}^{\int x^{-1}a_1dx}, \mm{e}^{\int x^{-1}a_2dx}).$ 		Hence,
		\begin{align*}
			H^{-1}\mm{diag}(\dfrac{1}{x}+a_1, -\dfrac{1}{x}+a_2)H=&\mm{diag}(\dfrac{1}{x}+a_1, -\dfrac{1}{x}+a_2);\\
			H^{-1}\vt(H)=&-\mm{diag}(a_1, a_2).
		\end{align*}

		Let $\bb f$ be $\lau Cx$-basis of $E$ obtained from $\{\mm e_1,\mm e_2\}$ by the base change $GH$. Then, 
		\[
		\mm{Mat}(\na,\bb f) =
		H^{-1}\mm{diag}(\dfrac{1}{x}+a_1, -\dfrac{1}{x}+a_2)H + H^{-1}\vt(H)=\mm{diag}(\dfrac{1}{x}, -\dfrac{1}{x}).\]
		
		Let $\na^{\rm rs}: E\to E$ be given by $\mm{Mat}(\na^{\rm rs}, \bb f)=0$. Then, $\mm{Mat}(P_\na,\bb f)=\mm{diag}(\dfrac{1}{x}, -\dfrac{1}{x})$. Moreover, $\na=\na^{\rm rs}+P_\na$ satisfies (i)-(iii) of Theorem \ref{section2--12}. Let now consider $\na_1^{\rm rs}$ on $M$ defined by $\mm{Mat}(\na_1^{\rm rs},\bb e)=\begin{pmatrix}
			0&0\\
			1&0
		\end{pmatrix}$. The decomposition $\na=\na_1^{\rm rs}+P_1$ satisfies (i) and (ii), but not (iii). It is obvious that $(g_1,g_2)\ne (0,0)$, so we get $P_1\ne P_\na$ since 
		 $$\mm{Mat}(P_1,\bb f)=\dfrac{1}{1-g_{2}g_{1}}\begin{pmatrix}
			\dfrac{1}{x}+\dfrac{g_{2}g_{1}}{x}&2\dfrac{g_{1}}{x}\mm{e}^{\int x^{-1}(a_1-a_2)dx}\\
			-2\dfrac{g_{2}}{x}\mm{e}^{\int x^{-1}(a_2-a_1)dx}&-\dfrac{1}{x}-\dfrac{g_{2}g_{1}}{x}
		\end{pmatrix}.$$
	\end{ex}

	\subsection{Deligne-Malgrange lattices}\label{section2--02} 		
	Fix $(E,\na)\in\mc(\lau Cx/C)$, $s\in \NN$ be its Turrittin index and  $\Ph\subset \XX_s$ the set of irregular values of $\na$. Recall that the Galois group $\mm G=\mm{Gal}(\lau C{x_s}/\lau Cx)$ is cyclic. Let $\na=\na^{\rm rs}+P_\na$ be the logarithmic decomposition of $\na$ and $\ce\in \mclog(\pos C{x}/C)$ be a \textit{Deligne-Manin lattice} lattice for $(E,\na^{\rm rs})$  having all exponents in $\tau$. In \cite{mochizuki09b}, Mochizuki calls  $\ce$ the \textit{Deligne-Malgrange} lattice. We recall the  TLJ-decomposition $(E_{(s)}, \na)=\bigoplus\limits_{\ph \in \Ph} (E_{\ph},\na_{\ph})$ of $(E, \na)$.
	The object $\ce_{(s)}=\pos C{x_s}\otu{\pos Cx} \ce$ is automatically a Deligne-Manin lattice for $(E_{(s)},\nabla^{\mm {rs}})$. Similar as in the proof of Theorem \ref{section2--12}, $\ce_{(s)}$  is closed by $\mm G$.

	\begin{prp}\label{section2--14} Let $(E,\na)\in\mc(\lau cx)$ and $\na=\na^{\rm rs}+P_\na$ its logarithmic decomposition. Let write $\na_{\ph}^{\rm rs}=\na^{\rm rs}|_{E_\ph}$ for each $\ph\in \Ph$. Then, $(E_{\ph},\na_{\ph}^{\rm rs})$ admits a Deligne-Manin lattice $\ce_{\ph}$ for  with respect to $\tau$ such that 
		$\ce_{(s)}=\bigoplus\limits_{\ph \in \Ph} \ce_{\ph}.$
	\end{prp}
	\begin{proof} For each $\ph\in\Ph$, let $\cf_{\ph}$ be a Deligne-Manin lattice of $(E_{\ph},\na_{\ph}^{\rm rs})$  with respect to $\tau$. 	By \cite[Theorem 6.8]{HdST22}, the identity  $\id_{E_{(s)}}$	is lifted to an isomorphism $$\ce_{(s)}\simeq\bigoplus\limits_{\ph \in \Ph} \cf_{\ph}$$ in $\mc_{\mm{log, s}}(\pos C{x_s}/C)$. 
	\end{proof}

	\section{Connections with an action of an Artinian ring}\label{section3--0} 
	Fix a local $C$-algebra $\La$ with maximal ideal $\g l$ and  $\dim_C\La$ is finite. 
	
	\begin{dfn} \label{section3--1} We define $\mc(\lau Cx/C)_{(\La)}$ as the category whose 
		\begin{enumerate}
			\item[{\it objects}]   are couples $\big((E,\na), \al\big)$ with $(E,\na)\in\mc(\lau Cx/C)$ and $\al:\La\to\mm{End}((E,\na))$ is a morphism of $C$-algebras, and an 
			\item[{\it arrow}] from $\big((E,\na), \al\big)$ to $\big((E,\na'), \al'\big)$ is a morphism $\ph:(E,\na)\to (E,\na')$ such that $\al'(\la)\circ\ph=\ph\circ\al(\la)$ for all $\la\in \La$.   
		\end{enumerate}
	\end{dfn}

	Given $(E,\na)\in \mc(\lau Cx/C)_{(\La)}$, $s\in \NN$ its Turrittin index. Denote by $\mm G$ the Galois group $\mm{Gal}(\lau C{x_s}/\lau Cx)$. Then $(E,\nabla)$  induces the  connection $\ov\nabla$ on $\ov E=E/\g lE$. Similarly, we have $(\ov{E_{(s)}},\ov\nabla)$ and an isomorphism 
	$\Phi: \ov E_{(s)}  \longrightarrow \ov{E_{(s)}}$
	in $\mc\big(\lau C{x_s}/C\big)$. Taking reduction modulo $\g l$ of $(E_{(s)},\na)=\bigoplus\limits_{\ph\in \Ph} (E_{\ph},\na_{\ph})$, we get 	
	\begin{equation}\label{section3--02} 
		\big(\ov{E_{(s)}}, \ov{\na}\big)=\bigoplus_{\ph\in \Ph} \big(\ov{E_{\ph}},\ov{\na_{\ph}}\big).
	\end{equation}
	\begin{prp}\label{section3--2}
		Let $(E,\na)\in \mc\big(\lau Cx/C\big)_{(\Lambda)}$  and $s\in\NN$ its Turrittin index. 
		Then, the following claims hold:
		\begin{enumerate}
			\item[(1)] If  $E_{\ph}\ne \{0\}$ then $\ov{E_{\ph}}\ne \{0\}$.
			\item[(2)] 
			The decomposition \eqref{section3--02} 
			is compatible to $\La$. 
		\end{enumerate}
	\end{prp}
	
	\begin{proof} Because $\g l\subset\Lambda$ is a nilpotent ideal of $\La$, so $\g lE_{\ph}\subsetneq E_{\ph}$. Hence, if $E_{\ph}\ne \{0\}$ then $\ov{E_{\ph}}\ne \{0\}$.
		Finally, the second claim then follows automatically by the very definition of the TLJ-decomposition. 
		%
		%
	\end{proof}
	As a consequence, we arrive at
	\begin{prp} \label{section4--4} Let $(E,\na)\in\mc\big(\lau Cx/C\big)_{(\La)}$. The following equality holds: 
		\[
		\text{Turrittin index of $(E,\na)$} \quad=\quad\text{Turrittin index of $(E/\g lE,\ov \na)$}.
		\]\qed
	\end{prp}

	Putting Theorem \ref{section2--12} and Propostion \ref{section3--2} together, we arrive at
	\begin{thm}
		\label{section3--3} Let $(E,\na)\in \mc(\lau Cx/C)_{(\La)}$ be given. Then there exists a unique decomposition up to isomorphism $\na=\na^{\rm rs} +P_\na$ having the following properties:
		\begin{itemize}
			\item[(i)] $(E,\na^{\rm rs})$ is an object in $\mcrs\big(\lau Cx/C\big)_{(\La)}$.
			\item[(ii)] $P_\na\in \mm{End}_{\lau \La x}(E)$ and it is diagonalizable over $\laus Cx$ as a $\lau Cx$--linear endomorphism.
			\item[(iii)] $(E,\na^{\rm rs})$ admits a TLJ-decomposition $(E,\na^{\rm rs}) = \bigoplus\limits_{\vr\in \Sigma}(E_\vr,\na_\vr)$ over $\lau Cx$, , where $\Sigma\subset C/\ZZ$. Moreover, for each ${\vr\in \Sigma}$, we have
			\begin{itemize}
				\item[(a)] $P_\na(E_\vr)\subset E_\vr$. 
				\item[(b)] There exists an $\laus Cx$--basis $(\bb e_\vr)$ of $\laus Cx\otu{\lau Cx}E_{\vr}$, $c_\vr\in C$ and $f_{\mm \vr 1},...,f_{\mm \vr \ell}\in x_s^{-1}C[x_s^{-1}]$ such that
				\[
				\begin{cases}
					\mm{Mat}(\na^{\rm rs}|_{E_{\vr}},\bb e_\vr)=\mm J^\flat_{\mm r_\vr}(c_\vr)\\
					\mm{Mat}({\na}|_{E_{\vr}},\bb e_\vr)=\mm{diag}(\mm J^\flat_{\mm r_{\mm \vr1}}(c_\vr+f_{\mm \vr 1}),...,\mm J^\flat_{\mm r_{\mm \vr \ell}}(c_\vr+f_{\mm \vr \ell})).
				\end{cases}
				\]
			\end{itemize}  
		\end{itemize}
	\end{thm}
	
	\begin{proof} According to Theorem \ref{section2--12}, we need only prove that $(E,\na^{\rm rs})$ is an object in $\mcrs(\lau Cx/C)_{(\La)}$. By Proposition \ref{section3--2}, each $(E_{\ph},\na_{\ph})$ is equipped with an action of $\La$.  Hence, we obtain that $(E,\na^{\rm rs})\in\mcrs(\lau Cx/C)_{(\La)}.$ 
	\end{proof}

	
	\begin{thm}\label{section3--4}
		Let $(E,\na)\in \mc(\lau Cx/C)_{(\La)}$ be given and let $\ce$ be the Deligne-Manin lattice of $(E,\nabla^{\mm rs})$ with respect to $\tau$. For each $\ph\in \Ph$, let $\na_{\ph}^{\rm rs}=\na^{\rm rs}|_{E_{\ph}}$. Then
		\begin{itemize}
			\item[(i)] $(\ce,\nabla^{\mm rs})$ is closed under the action of $\La$.
			\item[(ii)]  There exists a decomposition $\ce_{(s)}=\bigoplus\limits_{\ph \in \Ph} \ce_{\ph}$ such that $\ce_{\ph}\in \mc_{\mm{log,s}}(\laus Cx/C)_{\La}$ is a Deligne-Manin lattice for $(E_{\ph},\na_{\ph}^{\rm rs})$ having all exponents in $\tau$.
		\end{itemize} 
	\end{thm}
	\begin{proof} 
		According to Theorems \ref{section3--3}, we need only prove $(\ce,\nabla^{\mm rs})\in\mclog(\pos Cx/C)_{(\La)}$. Using \cite[Theorem 7.8]{HdST22}, the morphism $$\mm{End}_{\mclog}( \ce,\na^{\rm rs})\aro \mm{End}_{\mcrs}(E,\na^{\rm rs})$$
		is bijective. By using \cite[Theorem 7.8-(2)]{HdST22}, $\ce$ is closed under the action of $\La$ on $(E,\na^{\rm rs})$, and therefore, $(\ce,\na^{\rm rs})\in \mclog(\pos C{x}/C)_{(\La)}$. 
	\end{proof}

	\section{Logarithmic decomposition for connections parameterized by a complete discrete valuation ring}\label{section4--0}
	
	We recall here the notation $\mc^\circ$ stands for the full subcategory of $\mc$ containing $R$-torsion free objects. For each $s\in\NN^*$, we get 
	\begin{equation*}\label{eq.20221213}
		{\rm dLog}_s= \left\{\sum_{n\ge0}a_nx_s^n\in\pos R{x_s}:\,a_0\in s^{-1}\ZZ\right\}\simeq s^{-1}\ZZ+x_s\poss Rx
	\end{equation*}
	and $\mathds X_s$ is the group of isomorphism classes of objects of rank one in $\mc^{\mm{o}}_s\big(\laus Rx/R\big)$.

	For each $f\in\lau R{x_s}$, let $\g L_f=(\lau R{x_s} \bb f,\mm d_f)$ be the connection of rank one over $\lau R{x_s}$ defined by $\mm d_f(\bb f)=f\bb f$.  
	Note that $\g L_f$ is isomorphic to $\g L_g$ if and only if $f-g\in {\rm dLog}_s$, so the group $\mathds X_s$ is isomorphic to $R[x_s^{-1}]/s^{-1}\ZZ$. For each $\ph\in R[x_s^{-1}]/s^{-1}\ZZ$, denote by $\mm p(\ph)=\mm p(f_\ph)$, where $f_\ph$ is a representative in $\laus Rx$ of $\ph$.

	Let $d$ be a positive integer, $R'\supseteq R$ be an extension of $R$, $f\in \laus {R'}x$ and $(E,\na)$ be an object in $\mc\big(\lau Rx/R\big)$. Denote by $E_{\laus {R'}x}=E\otu{\lau Rx}\laus {R'}x$ and 
	$$\ker_{\laus {R'}x}\big(\na-f\big)^d:=\big\{e\in E_{\laus {R'}x}:\big(\na-f\big)^de=0\big\}.$$
	\subsection{The logarithmic decomposition of a connection}

	Let $(E,\na)\in\mc^{\mm{o}}\big(\lau Rx/R\big)$ and $s$ be a positive integer. Let us consider the following notions. 
	\begin{dfn}\label{04.05.2022--00}
		We say that $(E,\na)$ admits a \textit{TLJ-decomposition} over $\lau R{x_s}$ if there exists a decomposition 
		\begin{equation}\label{TLJ-decomp.2402}
			(E_{(s)},\na)=\bigoplus\limits_{\ph\in \Ph} (E_{\ph},\na_{\ph}),
		\end{equation}
		where $\Ph\subset R[x_s^{-1}]/s^{-1}\ZZ$ is a finite set and $(E_{\ph},\na_{\ph})\subset (E_{(s)},\na)$ is the subconnection generated by $\ker_{\laus Rx} (\na-f_\ph)^{q_\ph}$ for some $f_\ph\in \lau R{x_s}$ and $q_\ph\in \NN$ enough large. 
		\begin{itemize}
			\item Each element of $\Ph$ is called a \textit{character} of $(E,\na)$.
			\item The smallest number $s\in\NN$ such that \eqref{TLJ-decomp.2402} exists is called \textit{the Turrittin index} of $(E,\na)$. In addition, $(E,\na)$ is \textit{unramified} if $s=1$.
		\end{itemize}		
	\end{dfn}

	For the article's aim, we introduce the following notion:
	\begin{dfn}\label{04.05.2022--01} We say that the connection $(E,\na)$ admits a \textit{Turrittin-Levelt-Jordan form} (TLJ-form) over $\lau R{x_s}$  if, for each $\ph\in \Ph$ is a character of $(E,\na)$, there exist an $\lau R{x_s}$--basis $(\bb e_\ph)$ of 	$E_{\ph}$ such that   $$\mm{Mat}(\na_{\ph}, \bb e_\ph)=\mm J^\flat_{\mm r_\ph}(f_\ph),$$ where $f_\ph$ is a representative in $\laus Rx$ of $\ph$. 
	\end{dfn}

	\begin{rmk}
		Let $(E,\na)\in \mc\big(\lau Cx/C\big)$ and $s\in \NN$ be its Turrittin index. By Remark \ref{TLJ-form-2402}, $(E,\na)$ has a TLJ-form over $\laus Cx$. However, this property is not true for connections over $\lau Rx$.
	\end{rmk}

	\begin{ex}
		Consider the connection $(E,\na)$ defined by $E=\lau Rx\mm e_1\op\lau Rx\mm e_2$ and  $\na(\mm e_1)=0; \na(\mm e_2)=t\mm e_1$. Because $\na(\mm e_1)=\na^2(\mm e_2)=0$, so $(E,\na)$ is generated by $\ker_{\lau Rx} \na^{2}$. Notice that $(E,\na)$ has no TLJ-form over $\lau Rx$, but $\big(E_{\lau Kx},\na\big)$ admits a TLJ-form over $\lau Kx$ because
		$\mm{Mat}(\na, \bb e)=\begin{pmatrix}
			0&1\\
			0&0
		\end{pmatrix}$ with $\bb e=\{\mm e_1, \mm e_1+\dfrac{1}{t}\mm e_2\}$.
	\end{ex}

	\begin{ex}
		Consider the connection $(E,\na)$ defined by $E=\lau Rx\mm e_1\op\lau Rx\mm e_2$ and $\na(\mm e_1)=\mm e_2; \na(\mm e_2)=t^2\mm e_1$. Because $(E,\na)|_{t=0}$ is an indecomposable connection over $\lau Cx$, so $(E,\na)$ has no TLJ-form over $\lau Rx$. However, $\big(E_{\lau Kx},\na\big)$ is diagonalizable over $\lau Kx$ and $\mm{Mat}(\na, \bb e)=\begin{pmatrix}
			t&0\\
			0&-t
		\end{pmatrix}$ with $\bb e=\{t\mm e_1+\mm e_2, -t\mm e_1+\mm e_2\}$.
	\end{ex}

	\begin{ex}
		Let $(E,\na)$ be  the connection defined by $E=\lau Rx\mm e_1\op\lau Rx\mm e_2$ and $\na(\mm e_1)=\mm e_1; \na(\mm e_2)=t\mm e_1+(t+\dfrac{1}{2})\mm e_2$. By computing, we obtain that 
		\[(E,\na)=(E_1,\na_1)\oplus (E_{t+\frac{1}{2}},\na_{t+\frac{1}{2}}),\]
		where $(E_1,\na_1)$ and $(E_{t+\frac{1}{2}},\na_{t+\frac{1}{2}})$
		are the connections generated by $\ker_{\lau Rx} (\na-1)$ and $\ker_{\lau Rx} (\na-t-\dfrac{1}{2})$, respectively. More specifically, it admits a TLJ-form over $\lau Rx$ because
		$\mm{Mat}(\na, \bb e)=\begin{pmatrix}
			1&0\\
			0&t+\dfrac{1}{2}
		\end{pmatrix}$ with $\bb e=\{\mm e_1, 2t\mm e_1+\mm e_2\}$.
	\end{ex}

	\begin{prp} \label{07052022--02}
		Let $(E,\na)\in\mc^{\mm{o}}\big(\lau Rx/R\big)$ admit a TLJ-form over $\laus Rx$. Let us fix $\ph\in \Ph$ a character of $(E,\na)$. Then, there is a representative $h_\ph\in R[x_s^{-1}]$ of $\ph$ and an $\lau R{x_s}$-basis $(\bb e'_\ph)$ of $E_{\ph}$ such that $$\mm{Mat}(\na_{\ph}, \bb e'_\ph)=\mm J^\flat_{\mm r_\ph}(h_\ph).$$ 
	\end{prp}
	\begin{proof}
		Let us write $f_\ph=\sum\limits_{j\le0}a_{\ph j}x_s^j+\sum\limits_{i\ge1}a_{\ph i}x_s^i$ with all coefficients $a_{\ph i}\in R$. By computing, there exists $u\in\lau R{x_s}^\times$ such that $\vt(u)u^{-1}=-\sum\limits_{i\ge1}a_{\ph i}x_s^i$. The matrix of $\na_\ph$ with respect to $u\bb e_\ph$ is 
		$$u^{-1}\vt(u)\mm I_{\mm r_\ph}+\mm J^\flat_{\mm r_\ph}(f_\ph)=\mm J^\flat_{\mm r_\ph}(h_\ph)$$
		which belongs to $\mm{M}_{\mm r_\ph}\big(R[x_s^{-1}]\big)$, where $h_\ph=\sum\limits_{j\le0}a_{\ph j}x_s^j$.
	\end{proof}

	\begin{dfn} Let $(E,\na)\in\mc^{\mm{o}}\big(\lau Rx/R\big)$ admit a TLJ-decomposition over $\lau R{x_s}$ with its characters $\{\ph_1,...,\ph_n\}\subset R[x_s^{-1}]/s^{-1}\ZZ$. Let $f_{\ph_1},...,f_{\ph_n}$ be representatives in $R[x_s^{-1}]$ of $\ph_1,...,\ph_n$, respectively. 	We say that $(E,\na)$ has no \textit{turning point} at $\g r$ if $f_{\ph_1},...,f_{\ph_n}$ and their differences are invertible in $\laus {R}x$.	
	\end{dfn}

	According to  \cite[Theorrem 17.5.3, p.119]{andre-baldassarri-cailotto20}, we obtain the following proposition on the existence of TLJ-forms for a class of connections over $\lau Rx$. 
	\begin{prp}
		Let $(E,\na)\in\mc^{\mm{o}}\big(\lau Rx/R\big)$ be a connection of rank $n$ over $\lau Rx$ that has $n$ distinct characters in $R[x^{-1}]/\ZZ$. Assume that $(E,\na)$ has no {turning point} at $\g r$. Then $(E,\na)$ admits a TLJ-form over $\lau R{x}$.
	\end{prp}
	
	\begin{proof} Let $\ph_1,...,\ph_n$ be the distinct characters in $R[x^{-1}]/\ZZ$ of $(E,\na)$ and $f_1,...,f_n$ be their representatives in $R[x^{-1}]$, respectively. By using \cite[Section 6.a]{Lev75}, the connection $(E_{\lau Kx},\na)$ is diagonalizable in $\mc\big(\lau Kx/K\big)$. Hence, it has a TLJ-decomposition
		\begin{equation}\label{descend-TLJ form-1}
			\big(E_{\lau Kx},\na\big)=\bigoplus_{i=1}^n(E_{\ph_i},\na_{\ph_i}),
		\end{equation}
	where the isotypical component $(E_{\ph_i},\na_{\ph_i})$ is generated by $\ker_{\lau Kx}(\na-f_i)^{n_i}$ for each $1\le i\le n$. Because $\dim_{\lau Kx}E_{\lau Kx}=n$, 
	each $(E_{\ph_i},\na_{\ph_i})$ is a connection of rank one over $\lau Kx$, hence it is generated by $\ker_{\lau Kx}(\na-f_i)$. 
	
	The connection $(E,\na)$ satisfies the descent condition \cite[Theorem 17.5.3]{andre-baldassarri-cailotto20}, so the TLJ-decomposition \eqref{descend-TLJ form-1} descends to a decomposition of $(E,\na)$:
	\begin{equation*}
		(E,\na)=\bigoplus\limits_{i=1}^n \lau Rx\ot\ker_{\lau Rx}(\na-f_i).
	\end{equation*}
	Consequently, there is an $\lau Rx$-basis $\bb e=\big\{e_1,...,e_n\big\}$ with $e_i\in \ker_{\lau Rx}(\na-f_i)$ of $E$ for each $1\le i\le n$. Hence, the matrix of $\na$ with respect to $\bb e$ is $\mm{diag}(f_1,...,f_n)$.
	\end{proof}

Let $(E,\na)\in \mcrs^{\mm{o}}(\lau Rx/R)$. According to \cite[Theorem 8.16]{HdST22}, there exists an object $(V,A)\in \bb{End}_R^{\mm{o}}$ such that $(E,\na)$ is isomorphic to the Euler connection $\ga\mm{eul}(V,A)$, where all eigenvalues of the reduction of $A$ modulo $\g r$ belong to $\tau$. 
By combining \cite[Section 4.2, Proposition]{humphreys72} and \cite[Theorem 3.3]{ene06} together, we arrive at
\begin{prp}
	Let $(E,\na)\in \mcrs^{\mm{o}}(\lau Rx/R)$ be defined by $(V,A)\in \bb{End}_R^{\mm{o}}$. Assume that all eigenvalues of $A$ and their differences are invertible in $R$. Then $(E,\na)$ admits a TLJ-decomposition over $\lau Rx$.
\end{prp}
\begin{proof} Let $\la_1,...,\la_m$ be the eigenvalues in $R$ of $A$ with multiplicities $q_1,...,q_m$. The characteristic polynomial of $A$ is $\prod\limits_{i=1}^m (T-\la_i)^{q_i}.$ Since $\big\{\la_i-\la_j,\la_k\big\}_{1\le i,j,k\le m;i\ne j}$ are invertible in $R$, so the polynomials $(T-\la_i)^{q_i}, (T-\la_j)^{q_j}$ and $T$ are pairwise coprime for all $1\le i,j\le m$ and $i\ne j$. Now apply the Chinese Remainder Theorem, there exists a unique polynomial $p(T)\in R[T]$ such that $p(T)=\la_i\quad (\mm{mod}\quad (T-\la_i)^{q_i})$ and $p(T)\equiv 0\quad (\mm{mod}\quad T)$		for all $1\le i\le m$. Set $A_s=p(A)$ and $A_n=A-A_s$, we have $[A_s,A_n]=0$. Because the elements $\big\{\la_i-\la_j\big\}_{1\le i,j\le m;i\ne j}$ are invertible in $R$, so we obtain that $1=\sum\limits_{i=1}^{m}\prod\limits_{j\ne i}\big(\frac{T-\la_j}{\la_i-\la_j}\big)^{q_i}$. 
Let $V_i=\ker (A-\la_i)^{q_i}$ for each $1\le i\le m$, we get $V=\bigoplus\limits_{i=1}^m V_i$. The congruence $p(T)=\la_i\quad (\mm{mod}\quad (T-\la_i)^{q_i})$ implies that $A_s|_{V_i}=\la_i\id_{V_i}$ for all $i$. By definition, each $A_n|_{V_i}$ is nilpotent, then so is $A_n$.

Finally, according to \cite[Theorem 3.3, p.79]{ene06}, there exists an $R$--basis $\bb v$ of $V$ such that $\mm{Mat}(A,\bb v)=\mm{diag}(H_{r_1}(\la_1),...,H_{r_m}(\la_m)),$ 
where each $H_{r_i}(\la_i)\in \mm{M}_{r_i}(R)$ is a upper triangular matrix having all entries on the main diagonal $\la_i$ for each $1\le i\le m$.
Therefore, $$\mm{Mat}(\na,1\ot \bb v)=\mm{diag}(H_{r_1}(\la_1),...,H_{r_m}(\la_m))$$ which shows that $(E,\na)$ has a TLJ-decomposition over $\lau Rx$.
\end{proof}




	\begin{thm}[Logarithmic decomposition]
		\label{prp-31.03--2}
		Let $(E,\na)$ be an object in the category $\mc^{\mm{o}}\big(\lau Rx/R\big)$ which admits a TLJ-form over $\lau Rx$. Then there exists a unique decomposition $\na = \na^{\rm rs}+P_\na$ 			such that the following conditions are satisfied:
		\begin{itemize}
			\item[(i)] $(E,\na^{\rm rs})$ belongs to $\mcrs\big(\lau Rx/R\big)$ and admits a TLJ-form over $\lau Rx$.
			\item[(ii)] $P_\na\in \mm{End}_{\lau Rx}(E)$ is diagonalizable with all eigenvalues in $x^{-1}R[x^{-1}]$.
			\item[(iii)] Let $(E,\na^{\rm rs}) = \bigoplus\limits_{\vr\in \Sigma}(E_\vr,\na_\vr)$ be the TLJ-decomposition equipped with a TLJ-form over $\lau Rx$ of $(E,\na^{\rm rs})$, where $\Sigma\subset R/\ZZ$. For each ${\vr\in \Sigma}$, we have
			\begin{itemize}
				\item[(a)] $P_\na(E_\vr)\subset E_\vr$. 
				\item[(b)] There exists an $\lau Rx$--basis $(\bb e_\vr)$ of $E_{\vr}$, $c_\vr\in R$ and $f_{\mm \vr 1},...,f_{\mm \vr \ell}\in x^{-1}R[x^{-1}]$ such that
				\[
				\begin{cases}
					\mm{Mat}(\na^{\rm rs}|_{E_{\vr}},\bb e_\vr)=\mm J^\flat_{\mm r_\vr}(c_\vr)\\
					\mm{Mat}({\na}|_{E_{\vr}},\bb e_\vr)=\mm{diag}(\mm J^\flat_{\mm r_{\mm \vr1}}(c_\vr+f_{\mm \vr 1}),...,\mm J^\flat_{\mm r_{\mm \vr \ell}}(c_\vr+f_{\mm \vr \ell})).
				\end{cases}
				\]
			\end{itemize}  
		\end{itemize}
	\end{thm}
	\begin{proof} \textit{Existence.}  By Proposition \ref{07052022--02}, there is $\Ph\subset R[x^{-1}]/\ZZ$ and a decomposition $(E,\na)=\bigoplus\limits_{\ph\in \Ph} (E_{\ph},\na_{\ph})$
	such that $$\mm{Mat}(\na_{\ph}, \bb e_\ph)=\mm J^\flat_{\mm r_\ph}(f_\ph)$$ for each $\ph\in \Ph$, where $f_\ph\in R[x^{-1}]$ satisfies $c_\ph=f_\ph-\mm p(\ph)\in \tau$ and $(\bb e_\ph)$ is an $\lau R{x}$--basis of	$E_{\ph}$. For each $\ph\in\Ph$, let $\na_\ph^{\rm rs}=\na_\ph-\mm p(\ph)$. Then, $(E_\ph,\na_\ph^{\rm rs})$ is an object in $\mcrs\big(\lau Rx/R\big)$. It is obvious to see that $$\mm{Mat}(\na_\ph^{\rm rs}, \bb e_\ph)=\mm J^\flat_{\mm r_\ph}(c_\ph).$$ 
	Let us define the connection $\na^{\rm rs}$ on $E$ which is defined by $\na_{\ph}^{\rm rs}$ on $E_\ph$ for each $\ph\in\Ph$.  It is not difficult to see that $(E,\na^{\rm rs})$ admits a TLJ-form over $\lau Rx$. It is obvious that $P_\na=\na-\na^{\rm rs}$ is a diagonalizable $\lau R{x}-$linear endomorphism of $E$ with all eigenvalues in $x^{-1}R[x^{-1}]$.

	Consider the TLJ-decomposition $(E,\na^{\rm rs}) = \bigoplus\limits_{\vr\in \Sigma}(E_\vr,\na_\vr)$ over $\lau Rx$ of $(E,\na^{\rm rs})$. Let us fix a representative $c_\vr$ in $\tau$ of $\vr$. Note that $E_\vr$ is generated by $\bigcup\limits_{k\ge 1}\ker_{\lau Rx}\big(\na-c_\vr\big)^k$  over $\lau Rx$, and moreover
		$$(E_{\lau Kx},\na^{\rm rs}) = \bigoplus\limits_{\vr\in \Sigma}(\lau Kx\otu{\lau Rx}E_\vr,\na_\vr)$$ is the TLJ-decomposition of $(E_{\lau Kx},\na^{\rm rs})$ over $\lau Kx$. On the other hand, the isotypical component $E_\ph$ is generated by $\bigcup\limits_{k\ge 1}\ker_{\lau Rx}\big(\na-c_\ph\big)^k$ over $\lau Rx$ for each $\ph\in \Ph$, and the TLJ-decomposition
		$$(E_{\lau Kx},\na^{\rm rs}) = \bigoplus\limits_{\ph\in \Ph}(\lau Kx\otu{\lau Rx}E_\ph,\na_\ph^{\rm rs})$$ of $(E_{\lau Kx},\na^{\rm rs})$ exists over $\lau Kx$, where $\na_{\ph}^{\rm rs}=\na^{\rm rs}|_{E_\ph}$ for each $\ph\in \Ph$.  Because all $c_\vr$ and $c_\ph$ belong to $\tau$, each $e\in E_\vr$ belongs to $(\lau Kx\otu{\lau Rx}E_\ph)\cap E=E_\ph$ for some $\ph\in\Ph$. Conversely, the uniqueness of the TLJ-decomposition in $\mc\big(\lau Kx/K\big)$ implies that the isotypical component $E_\ph$ is contained in $E_\vr$. Therefore, 
		$\bigoplus\limits_{\vr\in \Sigma}(E_\vr,\na_\vr)$ and $\bigoplus\limits_{\ph\in \Ph} (E_{\ph},\na_{\ph}^{\rm rs})$ coincide. By construction, $(E,\na^{\rm rs})$ and $P_\na$ satisfy (i)-(iii).

		\textit{Uniqueness.} Assume that $\na=\dot\na +\dot P$ and $\na=\na^{\rm rs}+P_\na$ are logarithmic decompositions of $(E,\na)$. By scalar extension of these decompositions, they are both logarithmic decompositions of $(E_{\lau Kx},\na)$ in the sense of Theorem \ref{section2--12}. Hence, the uniqueness of the logarithmic decomposition implies that
		\[(E_{\lau Kx},\dot\na)=(E_{\lau Kx},\na^{\rm rs}).\]
		By restricting to $E$ we obtain that $\dot\na=\na^{\rm rs}$. The uniqueness is proved.
	\end{proof}

	\begin{prp}\label{prp.20220715--01} Let $(E,\na)\in\mc^{\mm{o}}\big(\lau Rx/R\big)$ admit a TLJ-form over $\lau Rx$ and $\ce\in\mclog\big(\pos Rx/R\big)$ be a Deligne-Manin lattice for $(E,\na^{\rm rs})$  with respect to $\tau$ (cf. \cite[Theorem 9.1]{HdST22}). Then, there exists a decomposition $\ce=\bigoplus\limits_{\ph \in \Ph} \ce_{\ph}$ such that $\ce_{\ph}$ 			is a Deligne-Manin lattice  for $(E_{\ph},\na_{\ph}^{\rm rs})$ having all exponents in $\tau$. \qed
	\end{prp}

	\subsection{Relation to limits}
	Denote by $(-)|_\ell$ the reduction modulo $\g r^{\ell+1}$ of objects in $\mc^{\mm{o}}\big(\lau Rx/R\big)$.
	
	\begin{dfn}\label{section4--3} We let $\mcrs\big(\lau Rx/R\big)^\wedge$ stand for the category whose \item [{\it objects}] are families $\{\big((E_\ell,\na_\ell),\ph_\ell\big)\}_{\ell\in \NN}$, where 
		$(E_\ell,\na_\ell)\in\mc(\lau Cx/C)_{(R_\ell)}$ and 
		$\ph_\ell:(E_{\ell+1},\na_{\ell+1})|_\ell\to (E_\ell,\na_\ell)$
		are isomorphisms in $\mc(\lau Cx/C)_{(R_\ell)}$;
		\item[{\it arrows}] between $\{\big((E_\ell,\na_\ell),\ph_\ell\big)\}_{\ell\in \NN}$ and $\{\big((E'_\ell,\na'_\ell),\ph'_\ell\big)\}_{\ell\in \NN}$ are  compatible  sequences  $\{(E_\ell,\na_\ell)\to (E'_\ell,\na'_\ell)\}$ 	in $\prod\limits_\ell\hh{\mc_{(R_\ell)}}{(E_\ell,\na_\ell)}{(E'_\ell,\na'_\ell)}$.
	\end{dfn}

	%

	
	Let $\{\big((E_\ell,\na_\ell),\ph_\ell\big)\}_{\ell\in \NN}$ be an object in $\mcrs\big(\lau Rx/R\big)^\wedge$. Then, each $(E_\ell,\na_\ell)$ admits a Deligne-Manin lattice $(\ce_\ell,\na_\ell)$ in $\mclog\big(\pos Cx/C\big)_{(R_\ell)}$ such that
	\[(\ce_{\ell+1},\na_{\ell+1})|_{\ell}\simeq (\ce_\ell,\na_\ell).\]
	Since $\pos Rx$ is $\g r$--adically complete, the family $\{(\ce_\ell,\na_\ell)\}_{\ell\in\NN}$ defines a logarithmic connection $(\ce,\na):=\varprojlim\limits_\ell (\ce_\ell,\na_\ell)$ over $\pos Rx$. The limit in the sense of  \cite[Theorem 9.6]{HdST22} of $\{\big((E_\ell,\na_\ell),\ph_\ell\big)\}_{\ell\in \NN}$ is defined by
	\[\varprojlim\limits_\ell (E_\ell,\na_\ell):=\ga\big(\ce,\na\big),\]
	where $\ga$ is the functor \eqref{gama-func}. According to \cite[Theorem 9.6]{HdST22}, for each regular singular connection $(E,\na)$ in $\mcrs\big(\lau Rx/R\big)$ we obtain that
	\[(E,\na)\simeq\varprojlim\limits_\ell (E,\na)|_\ell.\]
	However, this isomorphism maybe not exist for connections in $\mc(\lau Rx/R)$. Let's mention {\cite[Counter-Example 9.3]{HdST22}}.

	\begin{ex}\label{section4--5}
		Let $E=\lau Rx \bb e$ and $\na{\bb e} =\dfrac{t}{x}\bb e$. For each $\ell \in\NN$, $(E,\na)|_\ell$  is trivial in $\mcrs\big(\lau Cx/C\big)_{(R_\ell)}$. Hence, $\{(E,\na)|_\ell\}_{\ell\in \NN}$ is trivial in $\mcrs\big(\lau Rx/R\big)^\wedge$. By taking the limit,  $\varprojlim\limits_{\ell} (E, \na)|_\ell$ is the trivial in $ \mcrs\big(\lau Rx/R\big)$. Hence, there \textit{doesn't exist} $(E,\na)\simeq\varprojlim\limits_\ell (E,\na)|_\ell$ because $(E,\na)$ is not regular singular.
	\end{ex}

	\begin{prp}
		\label{section4--7} 
		Let $(E,\na)\in\mc\big(\lau Rx/R\big)$ and $\na|_\ell^{\rm rs}$ is the regular singular part of $\na|_\ell$ for each $\ell\in \NN$. Then there exists a regular singular connection $(\dot E,\dot \na)$ in $\mcrs\big(\lau Rx/R\big)$ such that 
		$(\dot E,\dot \na)|_\ell = (E|_\ell,\na|_\ell^{\rm rs}).$
	\end{prp}
	\begin{proof}
		According to Theorem \ref{section3--3}, we obtain that $\{(E|_\ell,\na|_\ell^{\rm rs})\}_{\ell \in \NN}$ is an object in $\mc_{\rm {rs}}(\lau Rx/R)^\wedge$ because $(E|_{\ell+1},\na|_{\ell+1}^{\rm rs})|_\ell=(E|_\ell,\na|_\ell^{\rm rs})$ for each $\ell\in \NN$. Hence, the limit $(\dot E,\dot \na)=\varprojlim\limits_\ell (E|_\ell,\na|_\ell^{\rm rs})$ is an object in $\mcrs(\lau Rx/R)$. 
	\end{proof}

	We now arrive at
	\begin{thm} \label{thm.31.03--1}
		Let $(E,\na)$ be an object in $\mc^{\mm{o}}\big(\lau Rx/R\big)$ which admits a TLJ-form over $\lau Rx$.  Then, $(E,\na^{\rm rs})$ and $(\dot E,\dot \na)$ are isomorphic.
	\end{thm}
	\begin{proof} According to Proposition \ref{07052022--02}, there exists $\Ph\subset R[x^{-1}]/\ZZ$ and a decomposition $(E,\na)=\bigoplus\limits_{\ph\in \Ph} (E_{\ph},\na_{\ph})$
		such that $\mm{Mat}(\na_{\ph}, \bb e_\ph)=\mm J^\flat_{\mm r_\ph}(f_\ph)$
		for each $\ph\in \Ph$, where $f_\ph\in R[x^{-1}]$ is a representative of $\ph$ and $(\bb e_\ph)$ is an $\lau R{x}$--basis of 	$E_{\ph}$. Let $P_{\na_\ph}=P_\na|_{E_\ph}$ and $\na_\ph=\na|_{E_\ph}$. It is easy to see that 
		\[\mm{Mat}(P_{\na_\ph}, {\bb e}_\ph)= \mm p(\ph)\mm I_n, \quad \mm{Mat}(\na_\ph^{\rm rs}, {\bb e}_\ph)=\mm J^\flat_{\mm r_\ph}(f_\ph- \mm p(\ph)).\]
		Let write $a=f_\ph -\mm p(\ph)=a_0+a_1t+a_2t^2+\cdots$ and $\mm p(\ph)=b_0+b_1t+ b_2t^2+\cdots$.
		Notice that $f_\ph\in R[x^{-1}]$ and $a=f_\ph -\mm p(\ph)\in R$, so $a_i\in C$ and $b_i\in x^{-1}C[x^{-1}]$. In addition, ${\rm ord}_x(b_i)\ge{\rm ord}_x(f_\ph)$ 	for all $i\in \NN$.

		For each $\ell\in\NN$, consider the $\lau Cx$--basis $(t^i{\bb e}_\ph)_{i\le \ell}:=({\bb e}_\ph, t{\bb e}_\ph,\ldots, t^\ell{\bb e}_\ph)$ of $E|_\ell$. Let $\mm A_{\ph,\ell}=\mm{Mat}(\na_\ph^{\rm rs}|_\ell, t^i{\bb e}_\ph)$ and $\mm B_{\ph,\ell}=\mm{Mat}(P_{\na_\ph}|_\ell, t^i{\bb e}_\ph)$. We have 
		\[\mm A_{\ph,\ell}=
		\begin{pmatrix}
			\mm J^\flat_n(a_0)&0&\cdots&0&0
			\\
			a_1\mm I_n&\mm J^\flat_n(a_0)&\cdots&0&0
			\\
			\vdots&\vdots&\ddots&\vdots&\vdots
			\\
			a_{\ell-1}\mm I_n&a_{\ell-2}\mm I_n&\cdots&\mm J^\flat_n(a_0)&0
			\\
			a_\ell \mm I_n&a_{\ell-1}\mm I_n&\cdots&a_1\mm I_n&\mm J^\flat_n(a_0) 
		\end{pmatrix}\]
		and
		\[\mm B_{\ph,\ell}=
		\begin{pmatrix}
			b_0\mm I_n&0&\cdots&0&0
			\\
			b_1\mm I_n&b_0\mm I_n&\cdots&0&0
			\\
			\vdots&\vdots&\ddots&\vdots&\vdots
			\\
			b_{\ell-1}\mm I_n&b_{\ell-2}\mm I_n&\cdots&b_0\mm I_n&0
			\\
			b_\ell \mm I_n&b_{\ell-1}\mm I_n&\cdots&b_1\mm I_n&b_0\mm I_n 
		\end{pmatrix}.\]
		Hence, we obtain that 
		\begin{align*}
			\mm{Mat}(\na_\ph|_\ell, t^i\bb e_\ph)=&\mm A_{\ph,\ell}+\mm B_{\ph,\ell}	=b_0\mm I_{(\ell+1)n}+\mm A_{\ph,\ell}+(\mm B_{\ph,\ell}-b_0\mm I_{(\ell+1)n}).
		\end{align*}

		According to Proposition \ref{section3--2} and Theorem \ref{section3--3}, each $E_\ph$ is invariant under $\dot\na|_\ell$ since $\dot\na|_\ell=\na|_\ell^{\rm rs}$. The restriction $\dot\na_{\ph}|_\ell$ of $\dot\na|_\ell$ is the connection on $E_\ph|_\ell$  defined by the matrix $\mm A_{\ph,\ell}+(\mm B_{\ph,\ell}-b_0I)$ with respect to $(t^i{\bb e}_\ph)_{i\le \ell}$. Let $({\bb e_\ph^{(1)}})$ be the $\lau Cx$-basis of $E_\ph|_\ell$ which is base changed from $(t^i{\bb e}_\ph)_{i\le \ell}$ by
		\[\mm P_1=
		\begin{pmatrix}
			\mm I_n&0&\cdots&0&0
			\\
			0&\mm I_n&\cdots&0&0
			\\
			\vdots&\vdots&\ddots&\vdots&\vdots
			\\
			0&0&\cdots&\mm I_n&0
			\\
			b'_\ell \mm I_n&0&\cdots&0&\mm I_n 
		\end{pmatrix},\]
		where $b'_\ell\in x^{-1}C[x^{-1}]$ satisfies $b_\ell+\vt(b'_\ell)=0$. The matrix of $\dot\na_{\ph}|_\ell$ with respect to $({\bb e_\ph^{(1)}})$ is $$\mm B_{\ph,\ell}^{(1)}:=P_1^{-1}(\mm B_{\ph,\ell}-b_0\mm I_{(\ell+1)n})P_1+P_1^{-1}\mm A_{\ph,\ell} P_1+P_1^{-1}\vt(P_1).$$ 
		Note that 
		
		$$\mm P_1^{-1}=
		\begin{pmatrix}
			\mm I_n&0&\cdots&0&0
			\\
			0&\mm I_n&\cdots&0&0
			\\
			\vdots&\vdots&\ddots&\vdots&\vdots
			\\
			0&0&\cdots&\mm I_n&0
			\\
			-b'_\ell \mm I_n&0&\cdots&0&\mm I_n 
		\end{pmatrix}.$$ 
	By computing, we obtain $\mm P_1^{-1}\mm A_{\ph,\ell}\mm P_1=\mm A_{\ph,\ell}; \mm P_1^{-1}(\mm B_{\ph,\ell}-b_0\mm I_{(\ell+1)n})\mm P_1=\mm B_{\ph,\ell}-b_0\mm I_{(\ell+1)n},$	and
		\[\mm P_1^{-1}\vt(\mm P_1)=\begin{pmatrix}
			0&0&\cdots&0&0
			\\
			0&0&\cdots&0&0
			\\
			\vdots&\vdots&\ddots&\vdots&\vdots
			\\
			0&0&\cdots&0&0
			\\
			\vt (b'_\ell)\mm I_n&0&\cdots&0&0 
		\end{pmatrix}.\]
		Therefore, the matrix of $\dot\na_{\ph}|_\ell$ with respect to the basis $({\bb e_\ph^{(1)}})$ is
		\[\mm B_{\ph,\ell}^{(1)}=\mm A_{\ph,\ell}+ \begin{pmatrix}
			0&0&\cdots&0&0
			\\
			b_1\mm I_n&0&\cdots&0&0
			\\
			\vdots&\vdots&\ddots&\vdots&\vdots
			\\
			b_{\ell-1}\mm I_n&b_{\ell-2}\mm I_n&\cdots&0&0
			\\
			0&b_{\ell-1}\mm I_n&\cdots&b_1\mm I_n&0 
		\end{pmatrix}.\]			
		Continuing this procedure, there exists an $\lau {C}x$--basis $({\bb e_\ph^{(\ell)}})$ of $E_\ph|_\ell$ such that $\mm{Mat}(\dot\na_{\ph}|_\ell, {\bb e_\ph^{(\ell)}})=\mm A_{\ph,\ell}$. On the other hand, $\mm A_{\ph,\ell}=\mm{Mat}(\na_\ph^{\rm rs}|_\ell, t^i{\bb e}_\ph).$ Hence, $$\mm{Mat}(\dot\na_{\ph}|_\ell, {\bb e_\ph^{(\ell)}})=\mm{Mat}(\na_\ph^{\rm rs}|_\ell, (t^i{\bb e}_\ph)_{i\le \ell}).$$

		By proceeding with all isotypical components, there exists a $\lau Cx$--basis $(\bb e^{(\ell)})$ of $E|_\ell$ such that $\mm{Mat}(\dot \na_\ell, {\bb e^{(\ell)}})=\mm{Mat}(\na^{\rm rs}|_\ell, t^i{\bb e}),$ 			where $\dot \na_\ell$ is the connection on $E|_\ell$ which is defined by $\dot\na_{\ph,\ell}$ on each $E_\ph|_\ell$. Therefore, we obtain that
		\[(\dot E,\dot \na)|_\ell\simeq (E,\na^{\rm rs})|_\ell.\]			
		Because the numbers ${\rm ord}_x(b_i)$ are bounded from below, so are ${\rm ord}_x(b'_i)$. By constructing, the base changes above are compatible with taking the limit. This shows that there exists an $\lau {R}x-$basis $(\bb e^{(\infty)})$ of $\dot E$ such that $\mm{Mat}(\dot\na, \bb e^{(\infty)})=\mm{Mat}(\na^{\rm rs}, \bb e)$. 
		In particular, we obtain that $(\dot E,\dot\na)\simeq (E,\na^{\rm rs}).$
	\end{proof}


\end{document}